\documentclass[a4paper,12pt]{article}
\usepackage[cp1251]{inputenc}
\usepackage[russian,english]{babel}
\usepackage{amsmath, amsthm, amsfonts, amssymb}

\makeatletter
\@addtoreset{equation}{section}
\makeatother
\makeatletter
\@addtoreset{section}{part}
\makeatother
\makeatletter
\@addtoreset{thm}{section}
\makeatother
\makeatletter
\@addtoreset{lem}{section}
\makeatother
\makeatletter
\@addtoreset{expl}{section}
\makeatother
\makeatletter
\@addtoreset{remk}{section}
\makeatother
\makeatletter
\@addtoreset{nsl}{section}
\makeatother
\makeatletter
\@addtoreset{defn}{section}
\makeatother

\newcommand{\cT}{{\mathcal T}}

\newcommand{\cF}{{\mathcal F}}

\newcommand{\ve}{\varepsilon}
\newcommand{\mbR}{{\mathbb R}}
\newcommand{\mbN}{{\mathbb N}}

\theoremstyle{plain}
\newtheorem{thm}{Theorem}[section]
\newtheorem{lem}{Lemma}[section]

\theoremstyle{definition}
\newtheorem{defn}{Definition}
[section]
\newtheorem{expl}{Example}
[section]
\theoremstyle{remark}

\theoremstyle{Corollary}
\newtheorem{cor}{Corollary}[section]
\theoremstyle{Proposition}

\begin{document}
\begin{center}
{\bf \Large{Hilbert-valued self-intersection local times for planar Brownian motion}}
\end{center}
\large
\begin{center}
Andrey Dorogovtsev, adoro@imath.kiev.ua\\
Olga Izyumtseva, olgaizyumtseva@gmail.com
\end{center}
{\bf Abstract}

In the paper E.B. Dynkin construction for self-intersection local time of planar Wiener process is extended on Hilbert-valued weights. In Dynkin construction the weight is bounded and measurable. Since the weight function can describe the properties of the media in which the Brownian motion moves, then relatively to the external media properties the weight function can be random and unbounded. In this article we discuss the possibility to consider the Hilbert-valued weights. It appears that the existence of Hilbert-valued renormalized by Dynkin self-intersection local time is equivalent to the embedding of the values of Hilbert-valued weight into Hilbert-Shmidt brick. Using A.A. Dorogovtsev sufficient condition of the embedding of compact sets into Hilbert-Shmidt brick in terms of isonormal process we prove the existence of Hilbert-valued renormalized by Dynkin self-intersection local time. Also using Dynkin construction we construct the self-intersection local time for the deterministic image of planar Wiener process.

{\bf keywords:} self-intersection local time, Dynkin renormalization, isonormal process, Hilbert-Shmidt brick.

60G15, 60J55, 60J65
\section{Introduction}
\label{Introduction}
The construction of renormalized self-intersection local time for planar Brownian motion was proposed by E.B. Dynkin in \cite{1}. Due to the articles of J.-F. Le Gall \cite{12} and A.S. Sznitman \cite{13} the renormalized self-intersection local times obtained geometrical meaning as the coefficients in asymptotic expansion of the area of planar Wiener sausage of small radius. In Dynkin construction the weight function $\rho$ was involved under assumption that this function is bounded and measurable. In this article we discuss the possibility to consider the Hilbert-valued weights. To give the motivation for such extension let us recall the main elements of Dynkin construction.
Let $w(t),\ t\in [0;1]$ be a planar Wiener process. Put
$$
\Delta_k=\{0\leq t_1\leq\ldots\leq t_k\leq 1\},
$$
$$
f_\ve(y)=\frac{1}{2\pi\ve}e^{-\frac{\|y\|^2}{2\ve}}, \ \ve>0, \ y\in\mbR^2.
$$
For bounded measurable weight function $\rho:\mbR^2\to\mbR$ define
$$
T^w_{\varepsilon,k}(\rho):=\int_{\Delta_k}\rho(w(t_1))\prod^{k-1}_{i=1}f_\ve(w(t_{i+1})-w(t_i))d\vec{t},
$$
and
$$
\cT^{w}_{\ve, k}(\rho)=\sum^k_{l=1}C^{l-1}_{k-1}
\left(
\frac{1}{2\pi}\ln\ve\right)^{k-l}
\int_{\Delta_l}\rho(w(t_1))\prod^{l-1}_{i=1}
f_\ve(w(t_{i+1})-w(t_i))d\vec{t}.
$$
E.B. Dynkin proved the following theorem.
\begin{thm}
\label{thm1.1}\cite{1} There exists the random variable $\cT^{w}_k(\rho)$ such that for any $p>0$
$$
\cT^{w}_k(\rho):=L_p\mbox{-}\lim_{\ve\to0}\cT^{w}_{\ve, k}(\rho).
$$
\end{thm}

\begin{defn}
\label{defn1.1}
The random variable $\cT^{w}_k(\rho)$ is said to be renormalized by Dynkin $k$-multiple self-intersection local time for $w$.
\end{defn}
One of the possible motivations for the introduction of the weight function $\rho$ is the following. The function $\rho$ can describe the properties of the media in which the Brownian motion moves. In particular, it will be shown in Section 3 of this article that the renormalization of the self-intersection local time for the process $F(w(t)),\ t\in[0;1]$ can be reduced to the renormalization of self-intersection local time for the process $w$ with the weight function
$$
\rho(u)=\frac{1}{|\det F^{\prime}(u)|^{k-1}},\ u\in\mbR^2.
$$
Let $x$   be a diffusion process in $\mbR^2$  described by the stochastic
differential equation
$$
\begin{cases}
dx(t)=a(x(t))ds+B(x(t))dw(t),\\
x(0)=x_0.
\end{cases}
$$
where the coefficients $a$  and $B$  are the jointly Lipschitz
functions and
$$
m_1I<B^*B<m_2I,
$$
with some positive constants $m_1, m_2$. In the article \cite{14} it was shown that
\begin{equation}
\label{eq1.1}
E \int_{\Delta_2}f_\ve(x(t_2)-x(t_1))d\vec{t}\sim\frac{1}{2\pi}\ln\frac{1}{\ve}\ E\int^1_0 \frac{1}{|\det
B(x(t))|}dt, \ \ve\to0+.
\end{equation}
This means that the renormalization of self-intersection local time for $x$ will involve the weight function
$$
\rho(u)=\frac{1}{|\det B(u)|^{k-1}},\ u\in\mbR^2.
$$
Another proof of this hypothesis can be found in \cite{13}, where the first term of the asymptotic expansion of the area of small diffusion sausage is exactly the same as in \eqref{eq1.1}. Relatively to the external media properties the weight function $\rho$ can be random and unbounded. In order to extend Dynkin construction on this case we suppose that $\rho$ and $w$ are independent, i.e. the function $\rho$ is defined on another probability space and square integrable. Therefore, $\rho$ can be considered as the function from $\mbR^2$ to $L_2(\Omega, \cF, P).$ Hence, $\rho$ can be treated as the Hilbert-valued function. Section 2 is devoted to this abstract case. It occurs that possibility to define the renormalization for Hilbert-valued $\rho$ is related to the question: How big is the set of values of the function $\rho$ ? We present an approach based on the conditions for compact set to be covered by Hilbert-Shmidt brick \cite{5}. It allows to consider the random weight functions with unbounded trajectories which essentially enlarge the field of application of Dynkin construction.
\section{Hilbert-valued self-intersection local times for planar Brownian motion}
In this Section we extend the action of operator $\cT^w_{k}$ on the Hilbert-valued functions $\rho.$ To do this we will reformulate Dynkin result in terms of the functional analysis. Consider
the family of linear operators from $L_\infty(\mbR^2)$ to $L_2(\Omega, \cF, P)$ which act by the rule
$$
\cT^w_{\ve, k}(\rho)=\sum^k_{l=1}C^{l-1}_{k-1}\Big(\frac{1}{2\pi}\ln\ve\Big)^{k-l}
\int_{\Delta_l}\rho(w(t_1))
\prod^{l-1}_{i=1}f_\ve(w(t_{i+1})-w(t_i))d\vec{t}.
$$
One can check that for any $k\geq1,\ 0<\ve<1$
$$
E\Big|
\sum^k_{l=1}C^{l-1}_{k-1}\Big(\frac{1}{2\pi}\ln\ve\Big)^{k-l}
\int_{\Delta_l}\rho(w(t_1))\cdot
$$
$$
\cdot
\prod^{l-1}_{i=1}f_\ve(w(t_{i+1})-w(t_i))d\vec{t}\Big|^2\leq
$$
\begin{equation}
\label{eq2.1}
\leq \frac{1}{\ve^2}(1+|\ln\ve|)^{2(k-1)}\|\rho\|^2_\infty.
\end{equation}
It follows from \eqref{eq2.1} that $\{\cT^w_{\ve,k},\ \ve>0\}$ is the family of continuous linear operators. Due to Theorem \ref{thm1.1} for any $\rho\in L_\infty(\mbR^2)$ and $k\geq 2$ there exists
$$
L_2\mbox{-}\lim_{\ve\to0}\cT^w_{\ve, k}(\rho)=:\cT^w_k(\rho).
$$
Therefore, the family of continuous linear operators $\{\cT^w_{\ve,k},\ \ve>0\}$
strongly converges to $\cT^w_k$ as $\ve\to0.$ Consequently,
$\cT^w_{k}: \ L_\infty(\mbR^2)\mapsto L_2(\Omega, \cF, P)$
is a continuous linear operator \cite{6}. It implies that there exists a constant $c_k>0$ such that
$$
E\cT^w_k(\rho)^2\leq
c_k\|\rho\|^2_\infty.
$$
Note that $\cT^w_k$ is not nonnegative operator. Really,
$$
E\cT^w_{ 2}(1)=\lim_{\ve\to0}E\cT^w_{\ve, 2}(1).
$$
Since
$$
E\cT^w_{\ve, 2}(1)=E\int_{\Delta_2}f_{\ve}(w(t_2)-w(t_1))d\vec{t}+\frac{1}{2\pi}\ln\ve
$$
$$
=\frac{1}{2\pi}\int_{\Delta_2} \frac{1}{t_2-t_1+\ve}d\vec{t}+\frac{1}{2\pi}\ln\ve
$$
$$
=\frac{1}{2\pi}(1+\ve)(\ln(1+\ve)-1)+\frac{1}{2\pi}\ve(\ln\ve-1)\to -\frac{1}{2\pi},\ \ve\to0,
$$
then $\cT^w_2(1)$ is not nonnegative random variable.
From now we suppose that the function $\rho$ takes its values in the real separable Hilbert space $H.$ Let $\{e_m,\ m\geq 1\}$ be the orthonormal basis in $H.$ Then one can try to define
\begin{equation}
\label{eq2.2}
\cT_k(\rho)=\sum^{\infty}_{m=1}\cT_k((\rho,e_m))e_m.
\end{equation}
To verify that the series \eqref{eq2.2} converges we have to put some conditions on the function $\rho.$ One of the possible conditions is the following.

{\bf Condition $(\ast)$}. There exists an orthonormal basis $\{e_m,\ m\geq 1\}$ in $H$ such that
$$
\sum^{\infty}_{m=1}\sup_{u\in\mbR^2}(\rho(u),e_m)^2<+\infty.
\eqno (\ast)
$$
\begin{thm}
\label{thm2.1}
Suppose that Condition $(\ast)$ holds. Then the series
$$
\cT^w_k(\rho)=\sum^{\infty}_{m=1}\cT_k(E(\rho,e_m))e_m
$$
converges in mean square in $H$.
\end{thm}
\begin{proof}
Note that
$$
E\cT_k((\rho, e_m))^2\leq
c_k\ \sup_{\mbR^2}(\rho(u),e_m)^2.
$$
It follows from Condition $(\ast)$ that
\begin{equation}
\label{eq2.3}
E\sum^\infty_{m=1}\cT_k((\rho, e_m))^2<+\infty.
\end{equation}
The relation \eqref{eq2.3} implies the convergence of series
$$
\sum^{\infty}_{m=1}\cT_k((\rho, e_m))e_m
$$
in mean square in $H$.
\end{proof}
Condition $(\ast)$ means that the values of $\rho$ are contained in some Hilbert-Shmidt brick. Let us recall this notion. Suppose that $\{e_k,\  k\geq1\}$ is an orthonormal basis in $H$ and $\{\ve_k,\ k\geq1\}$ is a sequence of nonnegative numbers such that
$$
\sum^{\infty}_{k=1}{\ve_k}^2<+\infty.
$$
\begin{defn}
\label{defn3.1}
The Hilbert-Shmidt brick corresponding to the pair $(\{e_k,\ k\geq1\},\{\ve_k,\ k\geq1\})$ is the set
$$
K(\{e_k,\ k\geq 1\},\{\ve_k,\ k\geq1\})=\{x\in H:\ \forall k\geq 1:\ |(x, e_k)|\leq\ve_k \}.
$$
\end{defn}
Note that the Hilbert-Shmidt brick is a compact set. If the function $\rho$ satisfies Condition $(\ast)$, then
$$
\rho(\mbR^2)\subset K(\{e_k,\ k\geq 1\},\{\sup_{u\in\mbR^2}|E\rho(u)e_k|,\ k\geq1\}).
$$
Let us consider results related to the embedding of compact sets into  Hilbert-Shmidt brick. The sufficient condition was given by A.A. Dorogovtsev in the joint paper with M.M. Popov \cite{5} in terms of isonormal Gaussian process.
\begin{defn}
\label{defn3.2}
The centered Gaussian process $\eta(u),\ u\in A\subset H$ with the covariance function
$$
E\eta(u)\eta(v)=(u,v)
$$
is said to be an isonormal process on $A.$
\end{defn}
\begin{thm}
\label{thm3.1}\cite{5} Let $A$ be a compact set in $H.$ If the isonormal process $\eta(u),\ u\in A$ has a continuous modification on $A,$ then $A$ can be covered by Hilbert Shmidt brick.
\end{thm}
Let $H_{\ve}$ be the number of elements of the minimal $\ve$-net for the compact set $A.$ It is known \cite{10} (Section 10.2, p. 77) that the sufficient condition for the continuity of an isonormal process on the compact set $A$ is the finiteness of Dudley integral, i.e.
$$
\int_{0+}\sqrt{\ln H_{\ve}}d\ve<+\infty,
$$
where the integral is taking over the positive part of some neighborhood of zero.
Consequently, the compact sets which satisfy Dudley condition can be covered by Hilbert-Shmidt brick. Here we prove some useful statements about the sets which can be covered by Hilbert-Shmidt bricks.
\begin{thm}
\label{thm3.2} Let the set $A$ in $H$ can be covered by Hilbert-Shmidt brick. If the set $C$ is bounded and finite-dimensional, then Minkovski sum $A+C$ can be covered by Hilbert-Shmidt brick.
\end{thm}
\begin{proof}
Consider Hilbert-Shmidt brick $K(\{e_k,\ k\geq1\},\{\ve_k,\ k\geq1\})$ and an element $h\in H.$ If
$$
h=\sum^{+\infty}_{k=1}\alpha_k e_k
$$
and
$$
C=\{x=th,\ t\in[-1;1]\}
$$
then
$$
K(\{e_k,\ k\geq1\},\{\ve_k,\ k\geq1\})+C
$$
$$
\subseteq K(\{e_k,\ k\geq1\},\{\ve_k+|\alpha_k|,\ k\geq1\}).
$$
\end{proof}
\begin{cor}
\label{cor3.1}
Let $Q$ be a projection in $H$ with the finite-dimensional kernel and $K:=K(\{e_k,\ k\geq1\},\{\ve_k,\ k\geq1\})$ be  Hilbert-Shmidt brick. Then $Q(K)$ can be covered by Hilbert-Shmidt brick.
\end{cor}
As a consequence one can obtain the next statement useful for a consideration of random fields.
\begin{thm}
\label{thm3.3} Suppose that for the function $\rho:\mbR^d\to H$  the isonormal random field $\zeta(u),\ u\in\mbR^d$ (i.e. $\zeta$ is the centered Gaussian field with $E\zeta(u)\zeta(v)=(\rho(u),\rho(v)),\ u,v\in\mbR^d$)
satisfies conditions

1) $\zeta$ has a continuous modification on $\mbR^d$

2) with probability one there exists the limit
$$
\zeta(\infty):=\lim_{\|u\|\to+\infty}\zeta(u).
$$
Then $\rho$ satisfies Condition $(\ast).$
\end{thm}
\begin{proof}
Since for the jointly Gaussian random variables convergence in probability implies convergence in mean square, then it follows from the conditions of the theorem that $\rho\in C(\mbR^d, H)$ and there exists the limit
$$
\rho(\infty):=\lim_{\|u\|\to+\infty}\rho(u).
$$
Consequently, $K=\rho(\mbR^d)\cup \rho(\infty)$ is a compact set in the space $H.$ Now consider the distance $\gamma$ in $\mbR^d$ which is continuous with respect to Euclidian distance and corresponds to the one-point compactification of $\mbR^d$ which is homeomorphic to $d+1$-dimensional sphere. Consider the new Hilbert space $H_1=H\oplus\mbR^{d+1}$ and a compact set
$$
\widetilde{K}=\{[\rho(u),\tilde{u}],\ u\in\mbR^d\}\cup\{[\rho(\infty),\widetilde{\infty}]\},
$$
where $\tilde{u}$ and $\widetilde{\infty}$ are elements of above mentioned compactification. Let us construct an isonormal process on $\widetilde{K}$ as follows. Take a standard Gaussian vector $\xi$ in $\mbR^{d+1}.$ Suppose that $\xi$ and $\zeta$ are independent. Define for $u\in\mbR^d$
$$
\kappa([\rho(u),\tilde{u}]):=\zeta(u)+(\tilde{u},\xi)
$$
and
$$
\kappa([\infty,\widetilde{\infty}]):=\zeta(\infty)+(\widetilde{\infty},\xi),
$$
where the second summands in the previous relations are the scalar products with the vector $\xi$ in $\mbR^{d+1}.$ Suppose that we took a modification of $\zeta$ which is continuous and such that there exists the limit when $\|u\|\to+\infty.$ Now take a sequence $\{x_n,\ n\geq 1\}$ of elements of $\widetilde{K}$ which converges to $x_0\in\widetilde{K}.$ If $x_n=[\rho(u_n),\tilde{u_n}],\ n\geq0,$ then $\tilde{u_n}\to\tilde{u_0},\ n\to+\infty.$ Consequently, $u_n\to u_0,\ n\to+\infty$ and $\rho(u_n)\to\rho(u_0),\ n\to+\infty$ (it can happen that $u_0=+\infty$ or some of $u_n=+\infty$). Hence, $\zeta(u_n)\to\zeta(u_0),\ n\to\infty$ and $(\tilde{u_n},\xi)\to(\tilde{u_0},\xi),\ n\to+\infty.$ Finally, $\kappa(x_n)\to\kappa(x_0),\ n\to+\infty.$ Due to Theorem \ref{thm3.1} the compact set $\widetilde{K}$ in $H$ can be covered by Hilbert-Shmidt brick. Using Corollary \ref{cor3.1} one can conclude that the set $K$ which is the projection of $\widetilde{K}$ on $H$ also can be covered by Hilbert-Shmidt brick. This means that $\rho$ satisfies Condition $(\ast).$ Theorem is proved.
\end{proof}
Let us return to the question of unboundedness of trajectories of the random field $\rho$ which satisfies Condition $(\ast).$ From this moment we take $H=L_2(\Omega, \cF, P).$
\begin{expl}
\label{expl3.1}
Let $\{f_n,\ n\geq1\}$ be a sequence of independent random variables such that $P\{f_n=n^{\frac{1}{6}}\}=\frac{1}{n},\
P\{f_n=0\}=1-\frac{1}{n}.$
Note that $f_n\in L_2(\Omega, \cF, P)$ and
$$
\|f_n\|^2=Ef^{2}_{n}=\frac{1}{n^{\frac{2}{3}}}\to0,\ n\to\infty.
$$
Now it follows from Borel-Cantelli lemma that
$$
\sup_{n\geq 0}f_n=+\infty\ \mbox{a.s.}
$$
Put for $t\in[n; n+1]$
$$
\rho_0(t)=(t-n)f_{n+1}+(n+1-t)f_n
$$
and
$$
\rho_0(+\infty)=0.
$$
Since
$$
\|f_n\|\geq\|f_{n+1}\|,
$$
then one can see that
\begin{equation}
\label{eq3.1}
\|\rho_0(t)\|\leq\|f_{[t]}\|\to0,\ t\to+\infty.
\end{equation}
Note that for $t_1, t_2\in[n; n+1]$
$$
\|\rho_0(t_1)-\rho_0(t_2)\|=
$$
$$
=\|(t_2-t_1)(f_{n+1}-f_n)\|\leq
$$
\begin{equation}
\label{eq3.2}
\leq 2\|f_n\||t_2-t_1|.
\end{equation}
The relation \eqref{eq3.2} implies that the function $\rho_0$ satisfies Lipschitz condition on $[1;+\infty)$. Also, it follows from \eqref{eq3.1} that $A=\rho_0([1;+\infty])$ is a compact set in $L_2(\Omega, \cF, P).$ Let us check that $A$ can be covered by Hilbert-Shmidt brick using Theorem \ref{thm3.3}. Consider the centered Gaussian process $\eta(t),\ t\in[1;+\infty]$ with the covariance function
$$
E\eta(t)\eta(s)=(\rho_0(t), \rho_0(s)).
$$
The process $\eta$ can be defined on another probability space then $(\Omega, \cF, P).$
Since the function $\rho_0$ satisfies Lipschitz condition, then
$$
E(\eta(t_2)-\eta(t_1))^2=\|\rho(t_2)-\rho(t_1)\|^2\leq
$$
\begin{equation}
\label{eq3.3}
\leq c|t_1-t_2|^2,\ c>0.
\end{equation}
It follows from \eqref{eq3.3} that $\eta$ has a continuous modification on $[1;+\infty).$ Let us check that $\eta(t)\to0,$ when $t\to+\infty$ a.s.
Note that
$$
\eta(t)=(\rho_0(t),\xi)=(t-[t])\eta([t]+1)+([t]+1-t)\eta([t]).
$$
Consequently,
$$
|\eta(t)|\leq(t-[t])|\eta([t]+1)|+([t]+1-t)|\eta([t])|
$$
\begin{equation}
\label{eq3.4}
\leq |\eta([t]+1)|\vee |\eta([t])|.
\end{equation}
It implies that it suffices to check that
$$
\eta(n)\to0,\ n\to\infty\ \mbox{a.s.}
$$
 Since
$$
E\eta(n)^2=\frac{1}{n^{2/3}},
$$
then
$$
\sum^\infty_{n=1}E\eta(n)^4<+\infty.
$$
Consequently,
\begin{equation}
\label{eq3.5}
\eta(n)\to0,\ n\to\infty\ \mbox{a.s.}
\end{equation}
Relations \eqref{eq3.4} and \eqref{eq3.5} imply that the process $\eta$ has a continuous modification on $[1;+\infty]$. It follows from Theorem \ref{thm3.3} that $\rho_0([1;+\infty])$ can be covered by Hilbert-Shmidt brick.
\end{expl}
\begin{expl}
\label{expl3.2}
Consider the function on $\mbR^2$ which is defined as
$$
f(u)=\int^{1}_{0}p_t(u)du,\ u\neq0,
$$
$$
f(0)=+\infty,
$$
where $p_t$ is a density of centered Gaussian vector with the covariance matrix $tI.$ It can be easily checked that for some constants $c_i,\ i=\overline{1,4}$
$$
f(u)\leq(c_1\ln(\|u\|^{-1})+c_2)1_{\|u\|\leq1}+c_3e^{-c_4\|u\|^2}1_{\|u\|>1}.
$$
Consequently, the random field
$$
\rho(u)=f(u-\xi),\ u\in\mbR^2,
$$
where $\xi$ is a standard Gaussian vector in $\mbR^2$ is square integrable. Also $\rho$ has unbounded trajectories whith probability one. Note that
$$
E(\rho(u_1)-\rho(u_2))^2=\int_{\mbR^2}(f(u_1-v)-f(u_2-v))^2\rho_1(v)dv
$$
$$
\leq\int_{\mbR^2}(f(u_1-v)-f(u_2-v))^2dv
$$
$$
=\int_{\mbR^2}|\hat{f}(\lambda)|^2|e^{i(u_2-u_1,\lambda)}-1|^2d\lambda
$$
$$
=4\int_{\mbR^2}\frac{(1-e^{-\frac{\|\lambda\|^2}{2}})^2}{\|\lambda\|^4}|e^{i(u_2-u_1,\lambda)}-1|^2d\lambda.
$$
$$
\leq c\int_{\mbR^2}\frac{(1-e^{-\frac{\|\lambda\|^2}{2}})^2}{\|\lambda\|^3}\|u_2-u_1\|d\lambda.
$$
The estimate on $f$ implies that the random field
$$
\rho_1(u)=e^{-\|u\|^2}\rho(u),\ u\in\mbR^2
$$
satisfies the following relation for $\|u\|>1$
$$
E(\rho(u_2)-\rho(u_1))^2
$$
\begin{equation}
\label{eq3.6}
\leq c_1\|u_2-u_1\|e^{-c_2\|u_1\|^2\wedge\|u_2\|^2}.
\end{equation}
Consider a centered Gaussian random field $\zeta(u),\ u\in\mbR^2$
such that
$$
E\zeta(u)\zeta(v)=E\rho_1(u)\rho_1(v),\ u,v\in\mbR^2.
$$
Due to Kolmogorov theorem $\zeta$ has a modification which is continuous on $\mbR^2.$ To check that $\zeta$ has a limit, when $\|u\|\to+\infty$ consider new random field
$$
\zeta_1(u)\zeta\Big(\frac{u}{\|u\|^2}\Big),\ 0<\|u\|\leq1,
$$
$$
\zeta_1(0)=0.
$$
The inequality \eqref{eq3.6} guarantees that for some $c_5$
$$
E(\zeta_1(u_2)-\zeta_1(u_1))\leq c_5\|u_2-u_1\|,\ \|u_1\|,\|u_2\|\leq 1.
$$
Consequently, $\zeta_1$ has a continuous modification. This means that a.s.
$$
\zeta(u)\to0,\ \|u\|\to+\infty.
$$
Hence $\zeta$ satisfies conditions of Theorem \ref{thm3.3} which implies that the random field $\rho_1$ satisfies Condition $(\ast).$
\end{expl}
\section{Deterministic image of Wiener process}
\label{section3}
In this section, using Dynkin construction let us try to construct the self-intersection local time for the  process $F(w(t)),\ t\in[0;1],$ where the deterministic $F:\mbR^2\to\mbR^2$ is the diffeomorphism. Also we will suppose that $F$ and its first derivative are bounded. Formal expression which describes k-multiple self-intersection local time for the precess $F(w)$ is the following
$$
T^{F(w)}_k:=\int_{\Delta_k}\prod^{k-1}_{i=1}\delta_0(F(w(t_{i+1}))-F(w(t_i)))d\vec{t}.
$$
To define an approximating family $T^{F(w)}_{\ve,k},$ we introduce the new delta family of functions related to the new distance generated by the map $F.$ Taking into account that the function $\rho$ in Dynkin construction depends only on the value of the process at the initial point we define the following delta family of functions
$$
f^{F}_{\ve}(v_1,\ldots,v_{k})=\frac{1}{|\det F^{\prime}(F^{-1}(v_1))|^{k-1}}\prod^{k-1}_{i=1}f_{\ve}(F^{-1}(v_{i+1})-F^{-1}(v_{i})).
$$
Let us check that $\{f^{F}_{\ve},\ \ve>0\}$ approximate the delta function.
\begin{lem}
\label{lem1.1}
For any $v_k\in\mbR^2,\ \varphi\in C_{b}(\mbR^{2(k-1)})$
$$
\int_{\mbR^2(k-1)}\varphi(v_1,\ldots,v_{k-1})f^{F}_{\ve}(v_1,\ldots,v_k)d\vec{v}\to \varphi(v_k,\ldots,v_k)
$$
as $\ve\to 0.$
\end{lem}
\begin{proof}
Put $F^{-1}(v_{i}))=u_i,\ i=\overline{i,k-1}.$ Then
$$
\int_{\mbR^2(k-1)}\varphi(v_1,\ldots,v_{k-1})f^{F}_{\ve}(v_1,\ldots,v_k)d\vec{v}=
$$
$$
\int_{\mbR^2(k-1)}\varphi(F(u_1),\ldots,F(u_{k-1}))\frac{\prod^{k-1}_{i=2}|\det F^{\prime}(u_i))|}{|\det F^{\prime}(u_1))|^{k-1}}
$$
\begin{equation}
\label{eq1.2}
\prod^{k-2}_{i=1}f_{\ve}(u_{i+1}-u_i)f_{\ve}(F^{-1}(v_k)-u_{k-1})d\vec{u}
\end{equation}
Changing the variables in the integrand
$$
u_{k-1}=z_{k-1},\ u_{k-1}-u_{k-2}=z_{k-1},\ldots,u_2-u_1=z_1
$$
one can obtain that \eqref{eq1.2} equals
$$
\int_{\mbR^2(k-1)}\varphi(F(z_{k-1}-z_{k-2}-\dots-z_1),\ldots,F(z_{k-1}-z_{k-2}))
$$
$$
\frac{\prod^{k-1}_{i=2}|\det F^{\prime}(z_{k-1}-z_{k-2}-\ldots-z_{i})|}{|\det F^{\prime}(z_{k-1}-z_{k-2})|^{k-1}}\prod^{k-2}_{i=1}f_{\ve}(z_i)
$$
$$
f_{\ve}(z_{k-1}-F^{-1}(v_k))d\vec{z}\to\varphi(v_k,\ldots,v_k)
$$
as $\ve\to0.$ Lemma is proved.
\end{proof}
It follows from Lemma \ref{lem1.1} that the approximating family for $T^{F(w)}_{\ve,k}$  can have the following representation
$$
T^{F(w)}_{\ve,k}:=\int_{\Delta_k}f^{F}_{\ve}(F(w(t_{1})),\ldots,F(w(t_{k}))d\vec{t}=
$$
$$
=\int_{\Delta_k}\frac{1}{|\det F^{\prime}(w(t_1))|^{k-1}}\prod^{k-1}_{i=1}f_{\ve}(w(t_{i+1})-w(t_{i}))d\vec{t}=
$$
$$
=\int_{\Delta_k}\rho(w(t_1)\prod^{k-1}_{i=1}f_{\ve}(w(t_{i+1})-w(t_{i}))d\vec{t},
$$
where $\rho(u)=\frac{1}{|\det F^{\prime}(u)|^{k-1}}.$
Then Dynkin renormalization is the following
$$
\cT^{F(w)}_{\ve, k}=\sum^k_{l=1}C^{l-1}_{k-1}
\left(
\frac{1}{2\pi}\ln\ve\right)^{k-l}
\int_{\Delta_l}\Big[\rho(w(t_1))\prod^{l-1}_{i=1}
f_\ve(w(t_{i+1})-w(t_i))\Big]d\vec{t}
$$
and the following statement holds.
\begin{thm}
\label{thm4.1}
For any $p\in\mbN$ there exists
$\cT^{F(w)}_{k}:=L_p\mbox{-}\lim_{\ve\to0}\cT^{F(w)}_{\ve, k}.$
\end{thm}
\par\bigskip\noindent
{\bf Acknowledgment.} The authors acknowledge financial support from the Deutsche Forschungsgemeinschaft (DFG) within the project "Stochastic Calculus and Geometry of Stochastic Flows with Singular Interaction" for initiation of international collaboration between the Institute of Mathematics of the Friedrich-Schiller University Jena (Germany) and the Institute of Mathematics of the National Academy of Sciences of Ukraine, Kiev.

\end{document}